\documentclass{asl}
\usepackage[dvips]{color}
\usepackage{color}

\title{Reducts of the Random Bipartite Graph}
\author{Yun Lu}

\address{Department of Mathematics\\
Kutztown University of PA\\
Kutztown, PA 19530}
\email{lu@kutztown.edu}

\thanks{The author thanks the anonymous referee for valuable feedback on an earlier version of this paper. This paper is a revision of the author's doctoral dissertation, written under the direction of Carol Wood at Wesleyan University. The author would like to thank Dr.Wood for her guidance and for helpful comments on earlier versions of this paper.}
\newtheorem{thm}{Theorem}[section]

\newtheorem{lem}[thm]{Lemma}

\newtheorem{defn}[thm]{Definition}

\newtheorem*{claima}{Claim A}
\newtheorem*{claimb}{Claim B}
\newtheorem*{claimc}{Claim C}

\begin{document}

\begin{abstract}
Let $\Gamma$ be the random bipartite graph, a countable graph with two infinite sides, edges randomly distributed between the sides, but no edges within a side. In this paper, we investigate the reducts of $\Gamma$ that preserve sides. We classify the closed permutation subgroups containing the group $Aut(\Gamma)^*$, where $Aut(\Gamma)^*$ is the group of all isomorphisms and anti-isomorphisms of $\Gamma$ preserving the two sides.  Our results rely on a combinatorial theorem of Ne\v{s}et\v{r}il-R\"{o}dl and a strong finite submodel property for $\Gamma$.
\end{abstract}
\maketitle

\section{Introduction}

 As in \cite{sT96}, a reduct of a structure $\Gamma$ is a structure with the same underlying set as $\Gamma$, for some relational language, each of whose relations is $\emptyset$-definable in the original structure. If $\Gamma$ is $\omega$-categorical, then a reduct of $\Gamma$ corresponds to a closed permutation subgroup in $Sym(\Gamma)$ (the full symmetric group  on the underlying set of $\Gamma$) that contains  $Aut(\Gamma)$ (the automorphism group  of $\Gamma$). Two interdefinable reducts are considered to be equivalent. That is, two reducts of a structure $\Gamma$ are equivalent if they have the same $\emptyset$-definable sets, or, equivalently, they have the same automorphism groups. There is a one-to-one correspondence between equivalence classes of reducts $N$ and closed subgroups of $Sym(\Gamma)$ containing $Aut(\Gamma)$ via $N\mapsto Aut(N)$ (see \cite{sT96}).

There are currently a few $\omega$-categorical structures whose reducts have been explicitly classified. In 1977, Higman classified the reducts of the structure $(\mathbb{Q}, <)$ (see \cite{gH77}). In 2008, Markus Junker and Martin Ziegler classified the reducts of expansions of $(\mathbb{Q}, <)$ by constants and unary predicates (see \cite{mJ05}). Simon Thomas showed that there are finitely many reducts of the random graph (\cite{sT91}) in 1991, and of the random hypergraphs (\cite{sT96}) in 1996. In 1995 James Bennett proved similar results for the random tournament, and for the random $k$-edge coloring graphs (\cite{jB95}). In this paper, we investigate the reducts of the random bipartite graph that preserve sides. We find it convenient to consider a bipartite graph in a language with two unary predicates (one side $R_l$, the other side $R_r$) and two binary predicates (edge $P_1$, not edge $P_2$). Equivalently, we analyze the closed subgroups of $Sym(R_l)\times Sym(R_r)$ containing $Aut(\Gamma)$, where $R_l$, $R_r$ denote the two sides of the random bipartite graph. Let $Aut(\Gamma)^*$ be a group of all isomorphisms and anti-isomorphisms preserving the two sides.  We classified all the closed subgroup of $Sym(R_l)\times Sym(R_r)$ containing $Aut(\Gamma)^*$. We have analyzed some closed groups between $Aut(\Gamma)$ and $Sym(\Gamma)$ but do not describe the results here since we do not  have a classification of all such.

%We first choose a suitable language for the bipartite graph.
\begin{defn}\label{bipartite}
A structure $G$= $(V^G, R_l^G, R_r^G, P_1^G, P_2^G)$, where $R_l^G,R_r^G\subseteq V^G$ and $P_1^G,P_2^G\subseteq R_l^G\times R_r^G$, is a \textit{bipartite graph} if it satisfies the following set of axioms:

$\exists x R_l(x)\wedge \exists x R_r(x)$;

$\forall x(R_l(x)\vee R_r(x))$;

$\forall x((R_l(x)\longrightarrow \neg R_r(x))\wedge (R_r(x)\longrightarrow \neg R_l(x)))$;

$\forall x\forall y((R_l(x)\wedge R_r(y))\longrightarrow (P_1(x, y)\vee P_2(x, y)))$;

$\forall x\forall y ((P_1(x, y)\longrightarrow (R_l(x)\wedge R_r(y)))\wedge (P_2(x, y)\longrightarrow (R_l(x)\wedge R_r(y))))$;

$\forall x\forall y((R_l(x)\wedge R_r(y))\longrightarrow ((P_1(x, y)\longrightarrow \neg P_2(x, y))\wedge (P_2(x, y)\longrightarrow \neg P_1(x, y))))$.

\end{defn}

In the rest of the paper, we will use the following notations: if $E=(a, b)\in R_l\times R_r$, then we call $(a, b)$ a cross-edge, and we say that $E$ has cross-type $P_i$ if $P_i$ holds for the pair $(a, b)$ for $i = 1, 2$. Furthermore, if $g\in Sym(\Gamma)$ and $E=(a, b)\in R_l\times R_r$, we denote $(g(a), g(b))$ by $g[E]$ . An $(m\times n)-$subgraph is a bipartite graph with $m$ vertices in $R_l$ and $n$ vertices in $R_r$. $Sym_{\{l, r\}}(\Gamma)$ denotes the group $Sym(R_l)\times Sym(R_r)$.

\begin{defn}\label{extension}
Let $n\in \mathbb{N}$. A bipartite graph satisfies the \textit{extension property} $\Theta_n$ if for any two disjoint subsets $X_{l1}$, $X_{l2}\in [R_l]^{\leq n}$, and any two disjoint subsets $X_{r1}$, $X_{r2}\in [R_r]^{\leq n}$,
\renewcommand{\labelenumi}{(\alph{enumi})}
\begin{enumerate}
\item there exists a vertex $v\in R_r$ such that $P_i(x, v)$ for every
$x\in X_{li}$ for $i=1, 2$; and
\item there exists a vertex $w\in R_l$ such that $P_i(w, x)$ for every $x\in X_{ri}$ for $i=1, 2$.
\end{enumerate}

\end{defn}
\begin{defn}
A countable bipartite graph, denoted by $\Gamma$, is \textit{random} if it satisfies the extension properties $\Theta_n$ for every $n\in \mathbb{N}$.
\end{defn}

The $\Theta_n$'s are first-order sentences, and the axioms in Definition \ref{bipartite} together with the  $\{\Theta_n\}_{n\in\mathbb{N}}$ form a complete and $\omega$-categorical theory. A random bipartite graph can be built by Fraisse-construction for bipartite graphs (see \cite{wH77}). It is countable and unique up to isomorphism. It is also easy to show that the random bipartite graph is homogeneous by a back-and-forth argument. In the rest of paper, we denote by $\Gamma$ the random bipartite graph.
%Notice that $\Gamma\backslash \{X\}\cong \Gamma$ for any finite $X\subseteq \Gamma$.

\begin{defn}
Let $\Gamma$ be the random bipartite graph and $A$ be a subset of $\Gamma$. A  bijection $\sigma: \Gamma\longrightarrow
\Gamma$ is a \textit{switch} with respect to $A$ if the following conditions are satisfied: \\
for all $(a, b)\in R_l\times R_r$ and $i=1, 2$, $P_i(a, b)\longleftrightarrow P_i(\sigma(a), \sigma(b))$ if and only if $|\{a, b\}\cap A|\neq 1$.
\end{defn}
Note that a switch on any finite set of vertices can be obtained by composing single-vertex switches.

\begin{defn}
Let $X\subseteq{\{l, r\}}$. The switch group \textit{$S_X(\Gamma)$} is the closed subgroup
of $Sym_{\{l, r\}}(\Gamma)$ generated as a topological group by
\renewcommand{\labelenumi}{(\arabic{enumi})}
\begin{enumerate}
\item Aut($\Gamma$); and
\item The set of all $\sigma\in Sym_{\{l, r\}}(\Gamma)$ such that $\sigma$ is a switch with respect to some $v\in R_i$, where $i\in X$.
\end{enumerate}
\end{defn}

Since $\Gamma$ satisfies the extension property $\Theta_n$ for $n\in \mathbb{N}$ and $S_{\{l, r\}}(\Gamma)$ is closed, we can construct $\rho\in S_{\{l, r\}} (\Gamma)$ which is a switch w.r.t. $R_l$. Observe that $\rho\in S_{\{l\}} (\Gamma)\cap S_{\{r\}} (\Gamma)$. Let $G^*$ be the closed group generated by $G$ and $\rho$. Then the group $S_X(\Gamma)^*$ is the same as the group $S_X(\Gamma)$ except when $X=\emptyset$. Notice $Aut(\Gamma)^*=S_{\emptyset}(\Gamma)^*$, which is a group of permutations that either preserve all cross-types on $R_l\times R_r$, or exchange all cross-types on $R_l\times R_r$. Also notice that $Aut(\Gamma)^*=S_l(\Gamma)\cap S_r(\Gamma)$.

We now state the main result of this paper.
\begin{thm}\label{classification}
If $G$ is a closed subgroup with $Aut(\Gamma)^*\leq G<
Sym_{\{l, r\}}(\Gamma)$, then there exists a subset $X\subseteq
\{l, r\}$ such that $G=S_X{(\Gamma)^*}$.
\end{thm}
That is, there are only finitely many closed subgroups of $Sym_{\{l, r\}}(\Gamma)$ containing $Aut(\Gamma)^*$: $Aut(\Gamma)^*, S_{\{l\}}(\Gamma), S_{\{r\}}(\Gamma), S_{\{l, r\}}(\Gamma)$, and $Sym_{\{l, r\}}(\Gamma)$. This theorem relies on a combinatorial theorem of Ne\v{s}et\v{r}il-R\"{o}dl and the strong finite submodel property of the random bipartite graph. It is still an open question whether there are finitely many closed subgroups between $Aut(\Gamma)$ and $Sym(\Gamma)$.

Here is how the rest of the paper is organized. In section 2, we study the relations preserved by the groups $S_X{(\Gamma)}$, where $X\subseteq\{l, r\}$. In section 3, we show that the random bipartite graph has the strong finite bipartite submodel property. In section 4, we employ a technique called $(m\times n)$-analysis for the random bipartite graph. These prepare us to give an explicit classification of the closed subgroups of $Sym_{\{, r\}}(\Gamma)$ containing $Aut(\Gamma)^*$ in the rest of the paper. In section 5, we prove the first part of Theorem \ref{classification}, which says that the closed subgroups of $S_{\{l, r\}}(\Gamma)$ containing $Aut(\Gamma)^*$ are $Aut(\Gamma)^*, S_{\{l\}}(\Gamma)$, and $S_{\{r\}}(\Gamma)$, and $S_{\{l, r\}}(\Gamma)$. In section 6, we proved the existence of some special finite subgraphs of $\Gamma$, which will be used in section 7. Then in section 7 we show there is no other proper closed subgroup between $S_{\{l, r\}}(\Gamma)$ and $Sym_{\{l, r\}}(\Gamma)$, which completes the proof of Theorem \ref{classification}.
%In section 7, we discussed the properties of some closed groups of $\Gamma$ that are not in $Sym_{\{l, r\}}(\Gamma)$.

%%%%%%%%%%%%%%%%%%%%%%%%%%%%%%%%%%%%%%%%%%%%%%%%%%%%%%%%%%%%%%%%%%%%%%%%%%%%%%%%%%%%%%%%%%%%%%%%%%%%%%%%%%%%%%%%%%%%%%%%%%%%%%%%%%%%%%%%%%%%%%%%%%%%%%%%%%%%%%%%%%%%%%%%%%%%%%%%%%%%%%%%%%%%%%%%%%%%%%%%%%%%

\section{Relations Preserved by Switch Groups}
In this section, we identify the relations preserved by the switch groups $S_{\{l\}}{(\Gamma)}, S_{\{r\}}{(\Gamma)}$ and $S_{\{l, r\}}{(\Gamma)}$. For convenience in discussing closures of $G\leq Sym_{\{l, r\}}(\Gamma)$, we let $\mathfrak{F}(G)=\{g\upharpoonright X \mid g\in G,X\in [\Gamma]^{<\omega}\}$.

\begin{defn}
Let $f\in Sym_{\{l, r\}}(\Gamma)$, and $S$ be a finite bipartite subgraph of $\Gamma$. We say $f$ preserves the parity of cross-types on $S$ if the number of $P_1$ cross-types in $S$ is even if and only if the number of cross-types in $f[S]$ is even.
\end{defn}

\begin{lem}\label{2and2}
$S_{\{l, r\}}(\Gamma)=\{\sigma\in Sym_{\{l, r\}}(\Gamma)\mid \sigma$
preserves the parity of cross-types in every $(2\times 2)$-subgraph
of $\Gamma\}$.
\end{lem}

\begin{proof}
It is easy to show that any $\sigma\in S_{\{l, r\}}(\Gamma)$ preserves the
parity of cross-types in every $(2\times 2)$-subgraph of
$\Gamma$. The other direction is proved as follows.

Suppose $\sigma\in Sym_{\{l, r\}}(\Gamma)$ preserves the parity of cross-types in every $(2\times 2)$-subgraph of
$\Gamma$.
Let B be an arbitrary $2\times 2$-subgraph of $\Gamma$. Since $\sigma$ preserves the parity of $P_i$'s for $i=l$ and $r$, only an even number of the cross-types can be changed. That is, $0$, $2$, or $4$ of the cross-types can be changed. We shall prove that in each case, there exists $\theta\in
S_{\{l, r\}}(\Gamma)$ such that $\theta\upharpoonright B=\sigma\upharpoonright B$.

Case 1: if none of the cross-types are changed, then there exists $\theta\in Aut(\Gamma)$ such that $\theta\upharpoonright B=\sigma\upharpoonright B$.

Case 2: if two of the cross-types are changed, then there exists $\theta$ which is either a switch w.r.t. one vertex or a switch w.r.t. two vertices of B such that  $\theta\upharpoonright B=\sigma\upharpoonright B$.

Case 3: if four of the cross-types are changed, then there exits $\theta$ which is a switch w.r.t. $R_l$ of $\Gamma$ (i.e. $\theta\in Aut(\Gamma)^*$) such that $\theta\upharpoonright B=\sigma\upharpoonright B$.

We then choose a vertex $v\in\Gamma\backslash {B}$ and let
$\phi=\theta^{-1}\circ \sigma\upharpoonright B\cup\{v\}$. We may assume $v\in R_l$. Note if
$E$ is a cross-edge in $B\cup\{v\}$ and $\phi$ does not preserve the cross-type on $E$, then $E=(v, u)$ for some $u\in R_r$. Also notice that $\theta$ and $\sigma$ both preserve the parity of
cross-types in $(2\times 2)$-subgraphs of $\Gamma$, hence so does $\phi$. Then it is easy to check that either for every $w\in B\cap R_r$, $P_i(v, w)\longrightarrow P_i(\phi(v), \phi(w))$; or for every $w\in B\cap R_r$, $P_i(v, w)\longrightarrow \neg P_i(\phi(v), \phi(w))$, where $i=1$ and $2$. Therefore $\phi\in\mathfrak{F}( S_{\{l,
r\}}(\Gamma))$, and so $\sigma\upharpoonright
B\cup\{v\}\in \mathfrak{F}(S_{\{l,
r\}}(\Gamma))$. Continuing in this manner for the vertices in $\Gamma\backslash {B\cup \{v\}}$, we see that for any finite bipartite graph
$S\subset\Gamma$, there exists an element $\theta_S\in \mathfrak{F}(S_{\{l,
r\}}(\Gamma))$ such that $\sigma\upharpoonright
S=\theta_S$. Thus $\sigma\in S_{\{l, r\}}(\Gamma)$, since $S_{\{l, r\}}(\Gamma)$ is
closed. This complete the proof of Lemma \ref{2and2}.
\end{proof}

Similarly, we can prove the following results.

\begin{lem}\label{1and2}
$S_{\{l\}}(\Gamma)=\{\sigma\in Sym_{\{l, r\}}(\Gamma)\mid\sigma$ preserves
the parity of cross-types in every $(1\times 2)$-subgraph of
$\Gamma\}$.
\end{lem}

\begin{lem}\label{2and1}
$S_{\{r\}}(\Gamma)=\{\sigma\in Sym_{\{l, r\}}(\Gamma)\mid \sigma$ preserves
the parity of cross-types in every $(2\times 1)$-subgraph of
$\Gamma\}$.
\end{lem}

%%%%%%%%%%%%%%%%%%%%%%%%%%%%%%%%%%%%%%%%%%%%%%%%%%%%%%%%%%%%%%%%%%%%%%%%%%%%%%%%%%%%%%%%%%%%%%%%%%%%%%%%%%%%%%%%%%%%%%%%%%%%%%%%%%%%%%%%%%%%%%%%%%%%%%%%%%%%%%%%%%%%%%%%%%%%%%%%%%%%%%%%%%%%%%%%%%%%%%%%%

\section{The Strong Finite Bipartite Submodel Property}

In this section, we define the Strong Finite Bipartite Submodel Property (SFBSP), inspired by the Strong Finite Submodel Property introduced by Thomas in \cite{sT96}, and we prove that the random bipartite graph has the SFBSP. This property provides a powerful tool in the later sections of this paper.
\begin{defn}
A countable infinite bipartite graph $\Gamma$ has the
\textit{Strong Finite Bipartite Submodel Property (SFBSP)} if $\Gamma=\bigcup_{i\in\mathbb{N}} \Gamma_i$ is a union of an increasing chain of substructures $\Gamma_i$ such that
\renewcommand{\labelenumi}{(\arabic{enumi})}
\begin{enumerate}
\item $\Gamma_i\subset \Gamma_{i+1}$ and $|\Gamma_i|=i$ for each $i\in\mathbb{N}$. In particular,
\begin{itemize}
\item if $i$ is even, then $|\Gamma_{i}\cap R_l|=|\Gamma_{i}\cap R_r|$;
\item otherwise, $|\Gamma_{i}\cap R_l|=|\Gamma_{i}\cap R_r|+1$.
\end{itemize}
\item for any sentence $\phi$ with $\Gamma\models \phi$, there exists $N\in \mathbb{N}$ such that $\Gamma_i\models \phi$ for all $i\geq N$.
\end{enumerate}
\end{defn}

\begin{thm}\label{sfsp}
The countable random bipartite graph $\Gamma$ has the SFBSP.
\end{thm}

Theorem \ref{sfsp} is a consequence of the Borel--Cantelli Lemma, as below:

\begin{defn}[\cite{sT96}]
If $\{A_n\}_{n\in \mathbb{N}}$ is a sequence of events in
a probability space, then $\bigcap_{n\in
\mathbb{N}}[\bigcup_{n\leq k\in \mathbb{N}}A_k]$ is the event that consists of realization of infinitely many of $A_n$, denoted by \textit{$\overline\lim A_n$}.
\end{defn}

\begin{lem}[Borel--Cantelli, \cite{pB79}]\label{bcl}
Let $\{A_n\}_{n\in \mathbb{N}}$ be a sequence of events
in a probability space. If $\sum_{n=0}^{\infty}{P(A_n)}<\infty$,
then $P(\overline\lim A_n)=0$.
\end{lem}

\begin{proof}[Proof of Theorem \ref{sfsp}]
 Since the extension properties $\Theta_n$'s axiomatize the random bipartite graph $\Gamma$ and $\Theta_i$ implies $\Theta_{i-1}$ for all $i\in \mathbb{N}$, for every sentence $\phi$ true in $\Gamma$, there exists some $k\in \mathbb{N}$ such that $\Theta_k$ holds if and only if $\phi$ holds. Let $\Omega$ be the probability space of all countable bipartite graphs $(S, R_l, R_r, P_1, P_2)$, where $| R_l|=| R_r|=\omega$ and every cross-edge $E\in R_l\times R_r$ has cross-type $P_1$ with probability $\frac{1}{2}$. For each $n\in \mathbb{N}$ with $n\geq k$,  let $S_n\in[S]^n$ such that if $n$ is even, then $|S_n\cap R_l|=\frac{n}{2}$, otherwise $|S_n\cap R_l|=|S_n\cap R_r|+1$. Let $A_n$ be the event that the induced graph on $S_n$ does not satisfy the extension property $\Theta_k$. Then by simple computation,
\begin{equation}
\sum_{n=0}^{\infty}{P(A_n)}=\sum_{m=0}^{\infty}{P(A_{2m})}+\sum_{m=0}^{\infty}{P(A_{2m+1})}
\leq 4\sum_{m=0}^{\infty}{{{m+1}\choose{k}}{{m+1-k}\choose{k}}(1-(\frac{1}{4})^k)^{m-2k}}
%({\frac{(2-(\frac{1}{4})^k)}{2}})^n*2*{{n}\choose{k}}^2\nonumber,
\end{equation}
where ${{n}\choose{i}}$ is the number of combinations of $n$ objects taken $i$ at a time. Let $C_m={{{m+1}\choose{k}}{{m+1-k}\choose{k}}(1-(\frac{1}{4})^k)^{m-2k}}$. Then $\lim_{m\to +\infty}\frac{C_{m+1}}{C_m}=1-(\frac{1}{4})^k<1$. By the ratio test for infinite series, we have $\sum_{m=0}^{\infty}{C_m}$ converges, and so does $\sum_{n=0}^{\infty}{P(A_n)}$. Thus by Lemma \ref{bcl}, $P(\overline\lim A_n)=0$. So there exists a bipartite graph $S\in\Omega$ and an integer $N$ such that for all $n\geq N$, the subgraph on $S_n\in [S]^n$ satisfies the extension property $\Theta_k$, and so $\phi$. Notice that the choice of $S$ ensures that $S$ is countable and satisfies all the axioms for the random bipartite graph. Hence $S$ is isomorphic to $\Gamma$. Then $\Gamma$ has the SFBSP, which completes the proof of Theorem \ref{sfsp}.
\end{proof}

In the rest of the paper, we often use the fact that $\Gamma$ has the usual finite submodel property. We will only use the \textit{strong} finite bipartite submodel property in section $7$.

%%%%%%%%%%%%%%%%%%%%%%%%%%%%%%%%%%%%%%%%%%%%%%%%%%%%%%%%%%%%%%%%%%%%%%%%%%%%%%%%%%%%%%%%%%%%%%%%%%%%%%%%%%%%%%%%%%%%%%%%%%%%%%%%%%%%%%%%%%%%%%%%%%%%%%%%%%%%%%%%%%%%%%%%%%%%%%%%%%%%%%%%%%%%%%%%%%%%%%%%%%%%%

%%%%%%%%%%%%%%%%%%%%%%%%%%%%%%%%%%%%%%%%%%%%%%%%%%%%%%%%%%%%%%%%%%%%%%%%%%%%%%%%%%%%%%%%%%%%%%%%%%%%%%%%%%%%%%%%%%%%%%%%%%%%%%%%%%%%%%%%%%%%%%%%%%%%%%%%%%%%%%%%%%%%%%55
\section{$(m\times n)$-analysis}
In \cite{sT96}, Thomas used a helpful tool called "m-analysis" to classify the reducts of the random hypergraphs. Using a similar approach, we give the definition of $(m\times n)$-analysis in this section, and we prove that if $f\in \mathfrak{F}(S_{\{l, r\}}{(\Gamma)})$ and if $|dom{f}|$ is sufficiently large, then $f$ has an $(m\times n)$-analysis. This rather technical concept will be used in the proof of Theorem \ref{classification}.

\begin{defn}
Let $m, n>2$. Suppose $f\in \mathfrak{F}(S_{\{l, r\}}{(\Gamma)})$
and $Z=dom{f}$ satisfies $|Z\cap R_l|\geq{m}$ and $|Z\cap R_r|\geq{n}$. An \textit{$(m\times n)$-analysis}
of $f$ consists of a finite sequence of elements ${f_0, f_1, ...,
f_s}\in \mathfrak{F}(S_{\{l, r\}}{(\Gamma)})$ satisfying the
following conditions:

\renewcommand{\labelenumi}{(\arabic{enumi})}
\renewcommand{\labelenumii}{(\alph{enumii})}
\begin{enumerate}
\item $f_0=\theta\circ f$ where $\theta\in \mathfrak{F}(Aut(\Gamma)^*)$;
\item For each $0\leq{j}\leq{s-1}$, there exist a finite $(m\times n)$-subgraph
$Y_j$ in $Z$, and
an element $\theta_j\in S_{\{l, r\}}{(\Gamma)}$ such that
\begin{enumerate}
\item $\theta_j$ is either an automorphism, or a switch with respect to some vertex $v_j\in Y_j\cap R_{i_j}$ where $i_j\in\{l, r\}$;
\item $\theta_j\upharpoonright{Y_j}=(f_j\circ{f_{j-1}}\circ ...
\circ{f_0})\upharpoonright{Y_j}$;
\item $f_{j+1}=\theta_j^{-1}\upharpoonright{ran({f_j}\circ...\circ{f_0}})$;
\end{enumerate}
\item $f_s\circ ... \circ{f_0}: Z\longrightarrow{\Gamma}$ is an
isomorphic embedding.
\end{enumerate}
\end{defn}

We now prove the existence of an $(m\times n)$-analysis for a given $f$.

\begin{thm}\label{analysis}
Let $m, n\in \mathbb{N}$ and $m, n>2$. For every $f\in\mathfrak{F}(S_{\{l, r\}}{(\Gamma)})$, there exists an integer $s(m, n)$ such that if $|dom{f}\cap R_i|\geq s(m, n)$ for $i=l$ and $r$, then there exists an $(m\times n)$-analysis of $f$. %(\color{red}{nontrivial f here, make sure %it won't affect the corresponding results})
\end{thm}

\begin{proof}
Let $f\in\mathfrak{F}(S_{\{l, r\}}{(\Gamma)})$ be such that $Z=dom f$ is a very large subset of $\Gamma$. By Ramsey's Theorem, there exists a large subset $S$ of $Z$ such that $S$ satisfies one of the following two conditions for every cross-edge $E$ in $S$, where $i=1, 2$:
\renewcommand{\labelenumi}{(\alph{enumi})}
\begin{enumerate}
\item $P_i(E)$ implies $P_i(f[E])$;
\item $P_i(E)$ implies $\neg P_i(f[E])$.
\end{enumerate}

We will construct a sequence of $f_i$'s as following.

 If $(a)$ holds, then we let $f_0=\theta\circ f$ where $\theta\in \mathfrak{F}(Aut(\Gamma)^*)$ is the identity map on $dom f$. Let $Y_0$ be an arbitrary $(m\times n)$-subgraph in $S$, and choose $\theta_0\in Aut(\Gamma)$ such that $\theta_0\upharpoonright S=f_0\upharpoonright S$. Define $f_1=\theta_0^{-1}\upharpoonright ran(f_0)$.

Next we choose $w_1\in Z\setminus S$ if it exists, and consider $f_1\circ f_0\upharpoonright S\cup\{w_1\}$. Since $f_1\circ f_0\in \mathfrak{F}(S_{\{l, r\}}{(\Gamma)})$ and $f_1\circ f_0\upharpoonright S$ is the identity map, $f_1\circ f_0\upharpoonright S\cup\{w_1\}$ is either an isomorphism or a switch with respect to $w_1$ by Lemma \ref{2and2}. Let $Y_1$ be an arbitrary $(m\times n)$-subgraph of $S\cup \{w_1\}$ containing $w_1$. Then there exists $\theta_1\in \mathfrak{F}(S_{\{l, r\}}{(\Gamma)})$ which is either an isomorphism or a switch with respect to $w_1$ and $\theta_1\upharpoonright S\cup\{w_1\}=f_1\circ f_0\upharpoonright S\cup\{w_1\}$. Define $f_2=\theta_1^{-1}\upharpoonright ran(f_1\circ f_0)$.

Continuing in this manner, for $0\leq j< s=|Z/S|$, we can find an $(m\times n)$-subgraph $Y_j$ of $Z$ and $\theta_j\in S_{\{l, r\}}{(\Gamma)}$ such that
\renewcommand{\labelenumi}{(\arabic{enumi})}
\begin{enumerate}
\item $\theta_j$ is either an isomorphism or a switch with respect to some vertex $w_j\in Y_j\cap R_{i_j}$ where $i_j\in\{l, r\}$;
\item $\theta_j\upharpoonright{Y_j}=(f_j\circ{f_{j-1}}\circ ...\circ{f_0})\upharpoonright{Y_j}$
\item $f_{j+1}=\theta_j^{-1}\upharpoonright{ran({f_j}\circ...\circ{f_0}})$;
\end{enumerate}
Also $f_s\circ ... \circ{f_0}: Z\longrightarrow{\Gamma}$ is an isomorphic embedding.

If $(b)$ holds, then there exists $\theta\in \mathfrak{F}(Aut(\Gamma)^*)$ with $dom (\theta)=ran (f)$, which exchanges all the cross-types on $\Gamma$. Let $f_0=\theta\circ f$. Hence $f_0\upharpoonright S$ is an isomorphism. The rest of the proof will be the same as in $(a)$.

Hence $f_0, f_1,\dots, f_{s}$ is an $(m\times n)$-analysis of $f$. This completes the proof of Theorem \ref{analysis}.

\end{proof}

\section{Closed Subgroups of $S_{\{l, r\}}(\Gamma)$ Containing $Aut(\Gamma)^*$}

In this section, we prove the first part of Theorem \ref{classification}, which says that the closed subgroups of $S_{\{l, r\}}(\Gamma)$ containing $Aut(\Gamma)^*$ are $Aut(\Gamma)^*$, $S_{\{l\}}(\Gamma)$, $S_{\{r\}}(\Gamma)$, and $S_{\{l, r\}}(\Gamma)$. Notice that in the rest of the paper, we only consider maps in $Sym_{\{l, r\}}(\Gamma)$. Hence from now on, we call $f\upharpoonright E$ is an isomorphism if $E=(a, b)$ is a cross-edge and $P_i (a, b)$ implies $P_i (f(a), f(b))$ for $i=1, 2$. We call $f\upharpoonright E$ is an anti-isomorphism if $E=(a, b)$ is a cross-edge and $P_i (a, b)$ implies $\neg P_i (f(a), f(b))$ for $i=1, 2$.

\begin{thm}\label{main1}
Suppose that $G$ is a closed subgroup with
$Aut(\Gamma)^*\leq{G}\leq{S_{\{l, r\}}(\Gamma)}$. Let $X$ be the
largest subset of $\{l, r\}$ such that
$S_{X}(\Gamma)^*\subseteq G$. Then $G\subseteq S_{X}(\Gamma)^*$, and so $G= S_{X}(\Gamma)^*$.
\end{thm}

In the rest of this section, we let $G$ be a closed subgroup with $Aut(\Gamma)^*\leq{G}\leq{S_{\{l, r\}}(\Gamma)}$, and $X$ be the largest subset of $\{l, r\}$ such that $S_{X}(\Gamma)^*\subseteq G$.

\begin{lem}\label{5.2}
Suppose that $g\in G$ is a bijection such that
  for every finite $T \subseteq \Gamma$ with $|T\cap R_i|\geq 2$ for $i=l$ and $r$,  we have $g\upharpoonright T\in \mathfrak{F}(S_X(\Gamma)^*)$. Then $g\in S_X(\Gamma)^*$.
\end{lem}

\begin{proof}
If $X\neq\emptyset$, from Lemma \ref{2and2}, Lemma \ref{1and2} and Lemma \ref{2and1}, we know that $g\upharpoonright T\in \mathfrak{F}(S_X(\Gamma))$ implies $g\in S_X(\Gamma)$. Then we are done. If $X=\emptyset$, then $S_{\emptyset}(\Gamma)^*=Aut(\Gamma)^*$. If $g\upharpoonright T\in \mathfrak{F}(Aut(\Gamma)^*)$, then $Aut(\Gamma)^*=S_{\{l\}}(\Gamma)\cap S_{\{r\}}(\Gamma)$ implies $g\upharpoonright T\in \mathfrak{F}(S_l(\Gamma))$ and $g\upharpoonright T\in \mathfrak{F}(S_r(\Gamma))$. Thus $g\in S_{\{l\}}(\Gamma)\cap S_{\{r\}}(\Gamma)$, and so $g\in Aut(\Gamma)^*$. This completes the proof of Lemma \ref{5.2}.
\end{proof}

Now let $g\in G$. Let $T\subseteq \Gamma$ be an arbitrary finite bipartite graph with $|T\cap R_i|\geq 2$ for $i=l$ and $r$. Then it will be sufficient to show that $g\upharpoonright T\in \mathfrak{F}(S_{X}(\Gamma)^*)$. To achieve this, we adjust $g$ repeatedly via composition with elements of $S_X(\Gamma)^*$ until we eventually obtain an element $h\in \mathfrak{F}(S_{\{l, r\}}(\Gamma))$ such that $h \upharpoonright T$ is an isomorphism. Our strategy is based upon the following lemma.

\begin{lem}\label{5.3}
Suppose $h\in \mathfrak{F}(S_{\{l, r\}}(\Gamma))$, and $U$, $T\subset dom (h)$ are two disjoint bipartite subgraphs such that for every cross-edge $E$ in $(T\cup U) \backslash T$, $h\upharpoonright E$ is an isomorphism. Then $h\upharpoonright T$ is an isomorphism.
\end{lem}

\begin{proof}
We prove this by contradiction. Suppose $h\upharpoonright T$ is not an isomorphism, then there exists a cross-edge $A \in [T]^2$ such that $h\upharpoonright A$ is not
an isomorphism. Let $W$ be a $(2\times 2)$-subgraph of $T\cup U$ such that $W\cap T=A$. By assumption, $h\upharpoonright E$ is an isomorphism for every cross-edge $E\in [W]^2 \backslash A$. Thus $h$ does not preserve the parity of the cross-types on the $(2\times 2)$-subgraph $W$, which contradicts Lemma \ref{2and2}. This completes the proof of Lemma \ref{5.3}.
\end{proof}

We shall make use of the following property of $X$.
\begin{lem}\label{H}
Let X be the largest subset of $\{l, r\}$ such that
$S_{X}(\Gamma)^*\subseteq G$. There exists a finite bipartite subgraph $H$ of $\Gamma$ satisfying:

For any $i\in{\{l, r\}}$, if there exist some vertex $v_i\in H\cap R_i$ and $g\in G$ such that $g\upharpoonright H$ is a switch w.r.t $v_i$, then $i\in X$.
\end{lem}

\begin{proof}
We prove the equivalent statement: there exists a finite bipartite subgraph $H$ of $\Gamma$ satisfying: if $i\in\{l, r\}$ and $i\notin X$, then for every $v_i\in H\cap R_i$ and every $g\in G$, $g\upharpoonright H$ is not a switch w.r.t $v_i$.

Since $i\in\{l, r\}$ and $i\notin X$, there exists a map $f$ which is a switch with respect to some vertex $a_i\in R_i$, but not in $G$. Otherwise the closed group generated by $Aut(\Gamma)$ and $f$ is $S_{\{i\}}(\Gamma)$, and so $S_{\{i\}}(\Gamma)=S_{\{i\}}(\Gamma)^*$ is a subgroup of $G$, a contradiction with the definition of $X$. Then $f\notin G$ implies that for every $g\in G$, $g$ is not a switch with respect to $a_i$. So there exists a finite set $A\subseteq \Gamma$ containing $a_i$ such that for every $g\in G$, $g\upharpoonright A$ is not a switch with respect to $a_i$

Since $\Gamma$ has the extension property, we have the following holds:

For every vertex $v_i \in R_i$, there exists a bipartite graph $A'\subseteq\Gamma$ containing $v_i$ which is isomorphic to $A$ mapping $v$ to $a_i$. This can be expressed by the first-order sentence $\sigma_i$.
If $\sigma$ is the sentence $\wedge_{i\notin X}{\sigma_i}$, then
$\Gamma\models\sigma$. Hence by Theorem \ref{sfsp}, there exists a finite
bipartite $H$ of $\Gamma$ such that
$H\models\sigma$. This $H$ satisfies our requirement, which completes the proof of Lemma \ref{H}.
\end{proof}

We shall also make use of a combinatorial theorem of Ne\v{s}et\v{r}il--R\"{o}dl, which is a generalization of Ramsey's Theorem. The following formulation, convenient for our use, is due to Abramson and Harrington (\cite{fA78}).

\begin{defn}[See \cite{sT96}]
A \textit{system of colors of length $n$}, $\alpha=(\alpha_1, ...,
\alpha_n)$ is an $n$-sequence of finite nonempty sets. An
\textit{$\alpha$-colored set} consists of a finite ordered set $X$
and a function $\tau: [X]^{\leq n}\longrightarrow
\alpha_1\cup...\cup\alpha_n$ such that $\tau(A)\in\alpha_k$ for each
$A\in[X]^k$ where $1\leq k\leq n$. For each $A\in[X]^{\leq n}$,
$\tau(A)$ is called the \textit{color} of $A$. An \textit{$\alpha$-pattern }is an $\alpha$-colored set whose
underlying ordered set is an integer.
%Note that each $\alpha$-colored set is isomorphic to an unique $\alpha$-pattern.
\end{defn}

\begin{thm}[Abramson--Harrington \cite{fA78}]\label{NN}
Given $n$, $e$, $M\in\mathbb{N}$, a system $\alpha$ of colors of
length $n$ and an $\alpha$-pattern $P$, there exists an
$\alpha$-pattern $Q$ with the following property. For any
$\alpha$-colored set $(X, \tau)$ with $\alpha$-pattern $Q$ and for
any function $F: [X]^e\longrightarrow M$, there exists $Y\subseteq
X$ such that $(Y, \tau\upharpoonright Y)$ has $\alpha$-pattern $P$
and such that for any $A\in[Y]^e$, $F(A)$ depends only on the
$\alpha$-pattern of $(A, \tau\upharpoonright A)$. (We say that such $Y$
is $F$-homogeneous).
\end{thm}%

\begin{proof}[Proof of Theorem \ref{main1}]
Let $X$ be the
largest subset of $\{l, r\}$ such that
$S_{X}(\Gamma)^*\subseteq G$. Suppose $g\in G$, and let $T\subseteq \Gamma$ with $|T\cap R_l|>2$ and
$|T\cap R_r|>2$. By Lemma \ref{5.2}, it is
enough to show now that $g\upharpoonright T\in \mathfrak{F}(S_{X}(\Gamma)^*)$. The proof of Theorem \ref{main1} proceeds via a sequence of claims.

Fix an ordering $\prec$ of vertices in $\Gamma$ such that $T$ is an initial segment of this ordering of $\Gamma$. For a
suitable system of colors $\alpha$, we define
an $\alpha$-coloring $\tau$ of $[\Gamma\backslash T]^{i\leq 2}$ by setting:

$\tau(A)=\tau(B)$ if and only if $|A|=|B|$ and the order-preserving bijection
$T\cup A\longrightarrow T\cup B$ is an isomorphism.

Now we define the partition function $F_g: [\Gamma\backslash T]^2\longrightarrow 2$ such that for $E\in [\Gamma\backslash T]^2$,
\begin{itemize}
\item $F_g(E)=1$ if $E\in [R_i]^2$ for $i=1, 2$; or if $E\in R_l\times R_r$ with $g\upharpoonright E$ is an isomorphism.
\item $F_g(E)=0$, otherwise.
\end{itemize}

Let $H$ be the finite bipartite graph given by Lemma \ref{H} and let
$m=|H\cap R_l|$, $n=|H\cap R_r|$. Since $\Gamma$ satisfies the extension properties, the following conditions hold.
\renewcommand{\labelenumi}{(\alph{enumi})}
\begin{enumerate}
\item $|\Gamma\cap R_i|\geq s(m,n)+|T|$ for $i=l$ and $r$, where $s(m, n)$ as in Lemma \ref{analysis};% for the sake of m-analysis
\item $\Gamma$ contains all different copies of $(2\times 2)$-graphs, each connecting to $T$ in all possible ways; %for the claim 4.9
\item $\Gamma$ contains isomorphic copies of $(m\times n)$-subgraph $H$ connecting to $T$ in all possible ways.

\item For every $v\in T$, there exists a finite bipartite subgraph $V\subseteq (\Gamma\backslash T)\cup \{v\}$ containing $v$ such that $V$ is isomorphic to the $(m\times n)$-subgraph $H$. %for all other copies of Y.
\end{enumerate}
These can be expressed as a first-order sentence $\sigma$. Since $\Gamma$ has SFBSP, there exists a finite subgraph $U\subset \Gamma\backslash T$ such that the conditions $(a)-(d)$ hold in $U$. Now let the $\alpha$-pattern P be the one derived from $(U, \tau\upharpoonright U)$. By Theorem \ref{NN} there exists $U'\subset \Gamma\backslash T$ such that $U'$ has the $\alpha$-pattern $P$. Thus $T\cup U'$ is isomorphic to $T\cup U$ sending $T$ to $T$. Furthermore, $U'$ is $F_g$-homogeneous. Now we will use the following Claims.

\begin{claima}\label{5.10}
Suppose that $X_1$, $X_2\subseteq U$ and that $|X_1\cap R_i|=|X_2\cap R_i|$ for $i=l$ and $r$. Let
$\phi: T\cup X_1\longrightarrow T\cup X_2$ be an order-preserving
bijection such that $\phi\upharpoonright E$ is an isomorphism for all $E\in[T\cup
X_1]^2\backslash [X_1]^2$. Then for all $E\in[X_1]^2$,

$g\upharpoonright E$ is an isomorphism if and only if $g\upharpoonright
\phi(E)$ is an isomorphism.
\end{claima}

\begin{proof}
We prove this by contradiction. We may assume that there exists some $E\in[X_1]^2$ such that $g\upharpoonright E$ is an isomorphism while $g\upharpoonright \phi[E]$ is
not. Since $U$ satisfies condition $(b)$, there exist $(2\times 2)$-subgraphs $V$, $W\subset U$, and $F\in [V]^2$,
$F'\in [W]^2$ with $\tau(E)=\tau(F)$ and $\tau(\phi[E])=\tau(F')$ satisfying the following condition.

There exists an order-preserving bijection $\alpha: T\cup V\longrightarrow T\cup W$ mapping $F$ to $F'$ such that for every $A\in[T\cup V]^2\backslash F$, $\alpha\upharpoonright A$ is an isomorphism.

In particular, $\tau(A)=\tau(\alpha(A))$ for all $A\in[V]^2\backslash F$. Since $U$ is $F_g$-homogeneous, it follows that for all $A\in[V]^2\backslash F$, $g\upharpoonright A$ is an isomorphism if and only if $g\upharpoonright\alpha(A)$ is an isomorphism. Since $\tau(E)=\tau(F)$ and $\tau(\phi[E])=\tau(F')$, we have $g\upharpoonright F$ is an isomorphism but $g\upharpoonright F'$ is not an isomorphism. Let $P=|\{A\in[V]^2\mid g\upharpoonright A$ is not an isomorphism$\}|$
and $Q=|\{A\in[W]^2\mid g\upharpoonright A$ is not an isomorphism$\}|$. Then $Q=P+1$ because
of the effect of $g$ on $F$ and $F'$. But by Lemma \ref{2and2}, $g\in S_{\{l, r\}}(\Gamma)$ implies $g$ preserve the parity of cross-types in $V$ and $W$. Thus $P$ and $Q$ must be even, which contradicts $Q=P+1$. This complete the proof of Claim A.
\end{proof}

\begin{claimb}\label{fU1}
$g\upharpoonright{U}\in\mathfrak{F}(S_X(\Gamma)^*)$.
\end{claimb}

\begin{proof}
Since $U$ satisfies the condition $(a)$, by Theorem \ref{analysis} there exists an $(m\times n)$-analysis of
$g\upharpoonright U$: $g_0, g_1, ...,
g_t\in\mathfrak{F}(S_{\{l, r\}}(\Gamma))$. That is, for each
$0\leq{j}\leq{t-1}$, there exists a finite $(m\times n)$-subgraph $Y_j$ in $U$ and an element
$\theta_j\in S_{\{l, r\}}(\Gamma)$ such that

\renewcommand{\labelenumi}{(\arabic{enumi})}
\begin{enumerate}
\item $g_0=\theta\circ g\upharpoonright U$ where $\theta\in \mathfrak{F}(Aut(\Gamma)^*)$;
\item $\theta_j$ is either an isomorphism or a switch w.r.t some vertex $a_j\in Y_j\cap R_{i_j}$ where $i_j\in\{l, r\}$;
\item $\theta_j\upharpoonright{Y_j}=(g_j\circ{g_{j-1}}\circ ...
\circ{g_0})\upharpoonright{Y_j}$;
\item
$g_{j+1}=\theta_j^{-1}\upharpoonright{ran({g_j}\circ...\circ{g_0})}$.
\item $(g_t\circ ... \circ{g_0}): U\longrightarrow{\Gamma}$ is an
isomorphic embedding.
\end{enumerate}

If all $\{i_0, ..., i_{t-1}\}\subseteq X$, then $g_0\upharpoonright
U\in \mathfrak{F}(S_X(\Gamma)^*)$, and so $g\upharpoonright
U\in \mathfrak{F}(S_X(\Gamma)^*)$. Otherwise, let $j$ be the least integer such that $i_j\notin X$ and the corresponding $\theta_j$ is a switch w.r.t. $a_j\in R_{i_j}\cap Y_j$. Note $\theta_0, ..., \theta_{j-1}\in S_X(\Gamma)^*$, which implies $g_1, ..., g_{j}\in\mathfrak{F}(S_X(\Gamma)^*)$. We prove this situation can not occur. Note
that $(g_j\circ ... \circ g_0)\upharpoonright Y_j=\theta_j\upharpoonright Y_j$ is a
switch w.r.t a vertex $a_j\in R_{i_j}\cap Y_j$.

%Since $\langle T\cup U_1;\prec\rangle$ contains the copies of $H$ connecting to $T$ in
%all the possible ways (by $\Theta_{N+m+n}$), there exists $Y\in[U_1\cap R_l]^m\times [U_1\cap R_r]^n$ satisfying

Since $U$ satisfies the condition $(c)$, there exist an $(m\times n)$-subgraph $H'\subseteq U$ which is an isomorphic copy of H, and a map $\phi$ satisfying that $\phi: T\cup Y_j\longrightarrow T\cup H'$ is an order-preserving bijection such that $\phi\upharpoonright E$ is an isomorphism for all $E\in[T\cup Y_j]^2\backslash [Y_j]^2$.

By Claim A, for every $E\in[Y_j]^2$, $g\upharpoonright E$ is
an isomorphism if and only if $g\upharpoonright \phi[E]$ is an isomorphism. Next we will show there exist $g_1^*, ...,
g_j^*\in\mathfrak{F}(S_X(\Gamma)^*)$ such that $g_j^*\circ ... \circ g_1^*\circ
g_0\upharpoonright H'$ is a switch w.r.t $\phi (a_j)$ of
 $H'$ in $R_{i_j}$. But then Lemma \ref{H} implies that $i_j\in X$, contrary to
our assumption. We define $g_l^*$ inductively for $1\leq l\leq j$ such that for all $E\in [Y_j]^2, g_l\circ ... \circ g_0\upharpoonright E$ is an isomorphism if and only if $g_l^*\circ ... \circ g_1^*\circ g_0\upharpoonright \phi[E]$ is an isomorphism.

Suppose $g_1^*, \dots, g_{l-1}^*$
have been defined, we now define $g_l^*$ for $1\leq l\leq j$:
\renewcommand{\labelenumi}{(\alph{enumi})}
\begin{enumerate}
\item If $\theta_{l-1}$ is an isomorphism, or if $\theta_{l-1}$ is a switch w.r.t. $a_{l-1}\in R_{i_{l-1}}$ but $a_{l-1}\notin Y_{j}$, then $g_l$ is an isomorphism on $g_{l-1}\circ ... \circ g_0 [Y_{j}]$, which is in $\mathfrak{F}(S_X(\Gamma))$. We define $g_l^*$ as the identity map on $ ran (g_{l-1}^*\circ ... \circ g_1^*\circ g_0)$;
\item Otherwise, $\theta_{l-1}$ is a switch w.r.t. $a_{l-1}\in R_{i_{l-1}}$ and $a_{l-1}\in Y_{j}$, then $g_l$ is a switch w.r.t $g_{l-1}\circ ... \circ g_0 (a_{l-1})\in R_{i_{l-1}}\cap g_{l-1}\circ ... \circ g_0 [Y_{j}]$. Then $g_l\in\mathfrak{F}(S_X(\Gamma))$. Let $\theta^*\in S_X(\Gamma)$ be a switch with respect to $g_{l-1}^*\circ ... \circ g_1^*\circ g_0(\phi(a_{l-1}))$, and define $g_l^*$ as $\theta^*\upharpoonright  ran (g_{l-1}^*\circ ... \circ g_1^*\circ g_0)$.
\end{enumerate}
This completes the proof of Claim B.

\end{proof}

Now choose $\psi_0\in S_X(\Gamma)^*$ such that
$\psi_0\upharpoonright{U}=g\upharpoonright{U}$, and let
$h_1=\psi_0^{-1}\circ{g}\upharpoonright{T\cup U}$. Then
$h_1\upharpoonright E$ is the identity for every $E\in [U]^2$.

Next, we choose a vertex $v_1$ in $T$. WLOG we let $v_1\in R_l$, and consider
$h_1\upharpoonright U\cup \{v_1\}$. Notice that if
$E\in[U\cup\{v_1\}]^2$ and $h_1\upharpoonright E$ is not an
isomorphism, then $v_1\in E$.

\begin{claimc}\label{fU2v}
$h_1\upharpoonright U\cup\{v_1\}\in\mathfrak{F}(S_X(\Gamma)^*)$.
\end{claimc}

\begin{proof}
Since $h_1\upharpoonright U=id$ and $h_1\in \mathfrak{F}(S_{\{l, r\}}(\Gamma))$, by Lemma \ref{2and2} $h_1$ preserves the parity of cross-types in every
$(2\times 2)$-subgraph of $U\cup\{v_1\}$. So $h_1\upharpoonright U\cup\{v_1\}$
is either an isomorphism or a switch with respect to $v_1$. We may assume $h_1\upharpoonright U\cup\{v_1\}$ is a switch with respect to $v_1$. Then there exists a switch $\psi_1\in S_{\{l\}}(\Gamma)$
such that $h_1\upharpoonright U\cup\{v_1\}=\psi_1\upharpoonright
U\cup\{v_1\}$, and for all $E\in[T\cup
U]^2$ with $v_1\notin E$, $\psi_1\upharpoonright E$ is an isomorphism.

If $l\in X$, then $\psi_1\in S_X(\Gamma)$ and so $\psi_1\in S_X(\Gamma)^*$, then we're done. Otherwise, we show that there will be contradiction. Since $U$ satisfies the condition $(d)$, there exists $(m\times n)$-subgraph $V$ in $U\cup\{v\}$ such that $v\in V$, and $V\simeq H$. Then $h_1\upharpoonright V$ is a switch with respect to $v_1\in R_l$. By Lemma \ref{H}, we have $l\in X$, a contradiction with our assumption. This completes the proof of Claim C.
\end{proof}

By Claim C, there exists $\psi_1\in S_X(\Gamma)^*$ is either an isomorphism or a switch w.r.t. $v_1\in R_i$ for $i\in X$ such that
\renewcommand{\labelenumii}{(\alph{enumi})}
\begin{enumerate}
\item $\psi_1\upharpoonright U\cup\{v_1\}=h_1\upharpoonright U\cup\{v_1\}$;
\item For all $E\in[T\cup U]^2$, if $v_1\notin E$, then $\psi_1\upharpoonright E$ is an isomorphism.
\end{enumerate}
Let $h_2=\psi_1^{-1}\circ h_1\upharpoonright T\cup U$, then for all $E\in [T\cup \{v_1\}]^2$, $h_2\upharpoonright E$ is an isomorphism.

Now choose a second vertex $v_2\in T\backslash\{v_1\}$. Arguing similarly as in Claim C, there exists $\psi_2\in S_X(\Gamma)^*$ which is either an isomorphism or a switch w.r.t. $v_2\in R_i$ for $i\in X$ such that
\renewcommand{\labelenumii}{(\alph{enumi})}
\begin{enumerate}
\item $\psi_2\upharpoonright U\cup\{v_2\}=h_2\upharpoonright U\cup\{v_2\}$;
\item For all $E\in[T\cup U]^2$, if $v_2\notin E$, then $\psi_2\upharpoonright E$ is an isomorphism.
\end{enumerate}
Note that such $\psi_2$ is an isomorphism for all the cross-edges $E$ such that $E\subseteq U$ or $E\cap T=\{v_1\}$. Thus when we next adjust $h_2$ to $h_3=\psi_2^{-1}\circ h_2\upharpoonright T\cup U$, we do not spoil the progress which we make with our earlier adjustments. Hence for all $E\in [T\cup \{v_1, v_2\}]^2\backslash \{v_1, v_2\}$, $h_3\upharpoonright E$ is an isomorphism.

By continuing in this fashion, we can deal with the other vertices in $T\backslash \{v_1, v_2\}$.
After $|T|$-1 steps, we obtain a map $h^*:
T\cup U\longrightarrow T\cup U$ such that
\renewcommand{\labelenumii}{(\alph{enumi})}
\begin{enumerate}
\item There exists $\psi^*\in S_X(\Gamma)^*$ such that $h^*=\psi^*\circ g\upharpoonright T\cup U$;
\item For all $E\in[T\cup U]^2\backslash [T]^2$, $h^*\upharpoonright E$ is an isomorphism.
\end{enumerate}

Now Lemma \ref{5.3} implies $h^*\upharpoonright T$ is an
isomorphism, hence $g\upharpoonright T={\psi^*}^{-1}\circ h^*\upharpoonright T\in\mathfrak{F}(S_X(\Gamma)^*)$. This completes the proof of Theorem \ref{main1}.
\end{proof}
\section{Some Special Finite Subgraphs of $\Gamma$}
In this section we show existence of some special finite bipartite subgraphs $\Gamma_{N_{G}}$ and $Z$. We will use the following two lemmas, each of which witness the fact that $G$ is a nontrivial reduct.

\begin{lem}\label{6.3}
Let $G$ be a proper closed subgroup of $Sym_{\{l, r\}}(\Gamma)$. There exists a finite bipartite subgraph $B_0$ of $\Gamma$ such that for every $g\in G$, there exist cross-edges $E_1, E_2$ in $B_0$ such that $P_1(g[E_1])$ and $P_2(g[E_2])$.
\end{lem}

\begin{proof}
Suppose no such $B_0$ exists, then for every finite bipartite subgraph $B$ of $\Gamma$, there exists some $g\in G$ such that either $P_1(g[E])$ for every cross-edge $E$ in $B$; or $P_2(g[E])$ for every cross-edge $E$ in $B$.

Express $\Gamma=\cup_{n\in\mathbb{N}}{\Gamma_n}$ as an union of an increasing
chain of finite bipartite subgraphs $\Gamma_n$.
There exists an infinite subset $I$ of $\mathbb{N}$ such that either for every $n\in I$, there
is $g_n\in G$ such that $P_1(g_n[E])$ for every cross-edge $E$ in $\Gamma_n$; or for every $n\in I$, there
is $g_n\in G$ such that $P_2(g_n[E])$ for every cross-edge $E$ in $\Gamma_n$.

We may assume the first situation holds. For any $(m\times n)$-subgraph
$C\subset\Gamma$ where $m, n\in \mathbb{N}$, there exists $N\in I$ such that $C\subseteq \Gamma_N$. Hence there exists some $g_c\in G$ such that $P_1(g_c[E])$ for every cross-edge $E$ in $C$. Then for any two $(m\times n)$-subgraphs $A,
B$ of $\Gamma$, we can find $\sigma\in Aut(\Gamma)$ sending $g_A[A]$ to $g_B[B]$. Then the map $f = g^{-1}_B \circ \sigma \circ g_A\in G$ and $f$ takes $A$ to $B$. But $A$ and $B$ are arbitrary $(m\times n)$-subgraphs of $\Gamma$, and so such $f's$ generate all of $Sym_{\{l, r\}}(\Gamma)$, a contradiction with the fact $G$ is a proper subgroup of $Sym_{\{l, r\}}(\Gamma)$. This completes the proof of Lemma \ref{6.3}.
\end{proof}

\begin{lem}\label{6.4}
Let $i\in \{l, r\}$ and $j\in \{1, 2\}$, G as above. There exists a finite bipartite subgraph $B_j^i$ of $\Gamma$ satisfying the following property for every $g\in G$:\\
$(\dagger)$ No vertex $v\in B_j^i\cap R_i$ has the property that for every cross-edge $E$ in $B_j^i$, $\neg P_j(g[E])$ if and only if if and only if $P_j(E)$ and $v\in E$.
\end{lem}

\begin{proof}
Fix $i$ and $j$. Let $m=|B_0\cap R_l|$ and $n=|B_0\cap R_r|$ for $B_0$ in Lemma
\ref{6.3}. We prove it by contradiction. Suppose there is no finite bipartite graph satisfying the property $(\dagger)$ for every $g\in G$. Then $B_0$ does not satisfy the property $(\dagger)$ for all $g \in G$,
then there exists some $g_0 \in G$ and $v_0 \in B_0$ such that $g_0
$ preserves the cross-types on all the cross-edges in $B_0$ except
those cross-edges $E$ where $P_j(E)$ and $v_0\in E$. Now compared
with $B_0$, $g_0[B_0]$ has fewer cross-edges with $P_j$ holding on
them. Note that $g_0[B_0]$ is finite, so it does not satisfy the property $(\dagger)$ by assumption. Similarly we can find
$g_1$ and $v_1 \in g_0[B_0]$ witnessing this failure, and such that
$g_1g_0[B_0]$ has even fewer cross-edges with $P_j$. Thus we can
find a sequence of elements of $G$ successively reducing the number
of instances of $P_j$, and finally we get their composite $g$
which, when applied to $B_0$, has eliminated all instances of $P_j$.
But this contradicts the property of $B_0$ in Lemma \ref{6.3}. Thus
some $(m\times n)$-subgraph must satisfy the requirement for $B_j^i.$
\end{proof}

Note that the following graphs exist in $\Gamma$:
\renewcommand{\labelenumi}{(\alph{enumi})}
\begin{enumerate}
\item the finite bipartite subgraph $B_0$ as in Lemma \ref{6.3};
\item the finite bipartite subgraph $B_i^j$ for $i\in \{l, r\}$ and $j\in \{1, 2\}$ as in Lemma \ref{6.4}.
\end{enumerate}
The existence of these finite subgraphs can be expressed by a first-order sentence $\sigma$, and $\Gamma\models \sigma$. By Theorem \ref{sfsp}, there exists $N_{G}\in \mathbb{N}$ such that for every $k\geq N_{G}$, $\Gamma_k$ satisfies $\sigma$.

In the rest of the section, we will prove the existence of a finite bipartite graph $Z\subset \Gamma$ having the properties that every $f\in G$ either preserves or interchanges cross-types on $Z$.

\begin{thm}\label{6.6}
Let $G$ be a proper closed subgroup of $Sym_{\{l, r\}}(\Gamma)$. There exists a finite bipartite subgraph $Z\subset \Gamma$ such that for every $f\in G$ and every cross-edge $E$ in $Z$, either $P_i(E)$ implies $P_i(f[E])$; or $P_i(E)$ implies $\neg P_i(f[E])$, where $i=1$ and $2$. That is, $f$ either preserves or interchanges cross-types on $Z$.
\end{thm}

\begin{proof}[Proof of Theorem \ref{6.6}]
Fix an ordering of the vertices of $\Gamma$. For a suitable system of colors $\alpha$, define an $\alpha$-coloring $\chi$ of $\Gamma^{i\leq 2}$ by setting:

$\chi(A)=\chi(B)$ if and only if $A, B\in [\Gamma]^{i\leq 2}$ and the bijection $A\rightarrow B$ is an isomorphism;

Let $P$ be the $\alpha$-pattern such that if $U$ is a finite bipartite $U$ of $\Gamma$ and $(U, \chi\upharpoonright U)$ has $\alpha$-pattern $P$, then $(U, \chi\upharpoonright U)\cong \Gamma_{N_{G}}$. By Theorem \ref{NN} there exists an
$\alpha$-pattern $Q$ such that for any
$\alpha$-colored set $(X, \chi\upharpoonright X)$ with $\alpha$-pattern $Q$ and for any partition $F: [X]^2\longrightarrow 2$, there exists $Z$ of $X$ such that $Z$ has the $\alpha$-pattern $P$, hence $Z\cong \Gamma_{N_{G}}$, and $(Z, \chi\upharpoonright Z)$ is $F$-homogeneous.

We define a particular partition $F: [X]^2\longrightarrow 2$ such that for every $E\in[X]^2$,
\begin{itemize}
\item $F(E)=1$ if $E\in [R_i]^2$ for $i=l, r$, or if $E$ is a cross-edge and $f$ preserves $P_j$ on $E$ for $j=1, 2$
\item $F(E)=0$ otherwise.
\end{itemize}
Then one of the following conditions must hold in $Z$ for every cross-edge $E$ where $i=1, 2$.

\renewcommand{\labelenumi}{(\arabic{enumi})$$}
\begin{enumerate}
\item $P_i(E)$ implies $P_i(f[E])$;
\item $P_i(E)$ implies $\neg P_i(f[E])$;
\item $P_1(f[E])$;
\item $P_2(f[E])$.
\end{enumerate}

 Note that $Z\cong \Gamma_{N_{G}}$, which contains $B_0$. This guarantees that only $(1)$ or $(2)$ hold for $Z$, as desired. This completes the proof of Theorem \ref{6.6}.
\end{proof}

\section{The Closed Groups between $S_{\{l, r\}}(\Gamma)$ and $Sym_{\{l, r\}}(\Gamma)$}%Groups between $S_{\{l, r\}}(\Gamma)$ and $Sym (R_l)\times Sym (R_r)$
In this section, we will prove the following Theorem:
\begin{thm}\label{6.1}
If $G$ is a closed subgroup such that $Aut(\Gamma)^*\leq G< Sym_{\{l, r\}}(\Gamma)$, then $G\leq S_{\{l, r\}}{(\Gamma)}$.
\end{thm}

 The SFBSP of $\Gamma$ will be used in the proof of Theorem \ref{6.1}. Recall that using Theorem \ref{sfsp}, we can express $\Gamma=\bigcup_{i\in\mathbb{N}} \Gamma_i$ as a union of an increasing chain of substructures $\Gamma_i$ such that\renewcommand{\labelenumi}{(\arabic{enumi})}
\begin{enumerate}
\item $\Gamma_i\subset \Gamma_{i+1}$ and $|\Gamma_i|=i$ for each $i\in\mathbb{N}$. In particular,
\begin{itemize}
\item if $i$ is even, then $|\Gamma_{i}\cap R_l|=|\Gamma_{i}\cap R_r|$;
\item otherwise, $|\Gamma_{i}\cap R_l|=|\Gamma_{i}\cap R_r|+1$.
\end{itemize}
\item for every sentence $\phi$ with $\Gamma\models \phi$, there exists $N\in \mathbb{N}$ such that $\Gamma_i\models \phi$ for all $i\geq N$.
\end{enumerate}

For the rest of this section, we fix $G$ be closed subgroup such that $Aut(\Gamma)^*\leq G< Sym_{\{l, r\}}(\Gamma)$. Let $X$ be the largest subset of $\{l, r\}$ such that $S_{X}(\Gamma)^*\subseteq G$, and so $X$ is also the largest subset of $\{l, r\}$ such that  $S_{X}(\Gamma)^*\subseteq G\cap S_{\{l, r\}}(\Gamma)$.  Note $G\cap S_{\{l, r\}}(\Gamma)$ is a closed subgroup of $S_{\{l, r\}}(\Gamma)$ containing $Aut(\Gamma)^*$, then by Theorem \ref{main1}, $G\cap S_{\{l, r\}}(\Gamma)=S_{X}(\Gamma)^*$.

\begin{proof}[Proof of Theorem \ref{6.1}]
We prove by contradiction. Assume $G$ is a closed subgroup with $Aut(\Gamma)\leq G< Sym_{\{l,
r\}}(\Gamma)$ but $G\nleqslant S_{\{l, r\}}(\Gamma)$. Then there exist a map $f\in G\backslash
S_{\{l, r\}}(\Gamma)$, and a
$2\times 2$-subgraph $Y$ of $\Gamma$ such that $f\upharpoonright Y$
does not preserve the parity of cross-types in $Y$. Let $Z\subset \Gamma$ be the finite bipartite
subgraph as in Theorem \ref{6.6}. Since $\Gamma$ is homogeneous,
there is $\phi\in Aut(\Gamma)$ such that $\phi(Z)=\Gamma_{N_G}$. Then there exists $s\in
\mathbb{N}$ such that $\phi(Y\cup Z)\subseteq \Gamma_s$. Let $M=\phi^{-1}[\Gamma_s]$. Then $Y\cup
Z\subseteq M$, and $\tau=\phi\upharpoonright M$ is an isomorphism from $M$ onto
$\Gamma_s$ with $\tau[Z]=\Gamma_{N_G}$.

For any $m$ with $N_G\leq m\leq s$, let $Z_m=\tau^{-1}[\Gamma_m]$ (Note $Z_{N_{G}}=Z$). By
Theorem \ref{6.6}, $f\upharpoonright Z_{N_G}\in \mathfrak{F}(S_{\{l, r\}}(\Gamma))$. Let $a$ be the greatest
integer such that $N_G\leq a\leq s$ and $f\upharpoonright Z_a\in  \mathfrak{F}(S_{\{l, r\}}(\Gamma))$. By the definition of $a$, Theorem \ref{main1} implies that there exists a map $\theta\in S_X(\Gamma)^*$
such that $f\upharpoonright Z_a=\theta\upharpoonright Z_a$. The existence of $Y\subseteq
M$ ensures that $a<s$. Suppose $Z_{a+1}=Z_a\cup\{v\}$. WLOG, let $v\in R_l$. We let
$f_1=(\theta^{-1}\circ f\circ\tau^{-1})\upharpoonright \Gamma_{a+1}$ and $w=\tau(v)$. By the maximality of $a$, $f\upharpoonright Z_{a+1}\notin  \mathfrak{F}(S_{\{l, r\}}(\Gamma))$. Thus $f_1\in \mathfrak{F}(G)\backslash \mathfrak{F}(S_{\{l,
r\}}(\Gamma))$.
%
%Note that for every $E\in[\Gamma_{a+1}]^2$, if $f_1\upharpoonright E$
%is not an isomorphism, then $w\in E$.%
%\begin{figure}[h]
%\centering%
%%\includegraphics[scale=0.4]{fig/Hd1.eps}
%\caption{the subgraph $\Gamma_{a+1}$}
%\label{fig: claim}
%\end{figure}

Fix an ordering $\prec$ of
$\Gamma_{a+1}$ such that $w$ is the initial element.  For a suitable system of colors $\alpha$,
define an $\alpha$-coloring $\eta$ of $[\Gamma\backslash \{w\}]^{i\leq 2}$ by setting: $\eta(A)=\eta(B)$ if and only if
the order-preserving bijection $\{w\}\cup A\longrightarrow \{w\}\cup B$ is an isomorphism.

 Let the $\alpha$-pattern $P$ be such that if $(S, \eta\upharpoonright S)$ has $\alpha$-pattern P,
 then $S\cup \{w\}\simeq \Gamma_{a+1}$. By Theorem \ref{NN} there exists a finite bipartite graph
 $Q\subseteq\Gamma\backslash \{w\}$ such that for any partition $F: [Q]^2\longrightarrow 2$,
 there exists $V$ of $Q$ such that there exits an isomorphism $\sigma: V\cup \{w\}\longrightarrow \Gamma_{a+1}$ sending $w$ to $w$. Furthermore, $(V,
 \eta\upharpoonright V)$ is $F$-homogeneous.
Now we define the partition function $F: Q\longrightarrow 2$ for every $a\in Q$
\begin{itemize}
\item $F(a)=1$ if $a\in R_r$ and $f_1\upharpoonright (w, a)$ is an anti-isomorphism;

\item $F(a)=0$ if $a\in R_l$, or $a\in R_r$ with $f_1\upharpoonright (w, a)$ is an isomorphism.
\end{itemize}

Let $U=V\cup \{w\}$. Then one of the following conditions must hold on $U$. %Since $(U_1, \eta\upharpoonright U_1)$ is
%$F$-homogeneous, for every $x\in U_1\cap R_l$, $F(x)=I$; and for every $x\in %U_1\cap R_r$, one of
%the following must hold:
%By the definition of $F$, each condition $(i)$ implies a corresponding
%condition $(i)'$ stating
%\renewcommand{\labelenumi}{(\arabic{enumi})$'$}
\begin{enumerate}
\item $f_1\circ \sigma$ is an isomorphism;
\item $f_1\circ \sigma$ is a switch with respect to $w$;
\item for all $E\in [U]^2$, $f_1\circ \sigma\upharpoonright E$ is not an
isomorphism if and only if $P_2(E)$ and $w\in E$;
\item for all $E\in [U]^2$, $f_1\circ \sigma\upharpoonright E$ is not an
isomorphism if and only if $P_1(E)$ and $w\in E$.
\end{enumerate}

Note $U\cong \Gamma_{a+1}$ and $\Gamma_{a+1}\supseteq \Gamma_{N_G}$, and $\Gamma_{N_G}$ contains an isomorphic copy of $B_1^l, B_2^l$, so $U$ contains isomorphic copies of $B_1^l$ and of $B_2^l$, which fail to obey the conditions $(3)$ and $(4)$. Thus only the conditions $(1)$ or $(2)$ holds in $U$, which implies that $f_1\circ \sigma\upharpoonright U \in
\mathfrak{F}(S_{\{l, r\}}(\Gamma))$, and so $f_1\in
\mathfrak{F}(S_{\{l, r\}}(\Gamma))$. This contradicts the fact that $f_1\notin \mathfrak{F}(S_{\{l, r\}}(\Gamma))$. This completes the proof of Theorem \ref{6.1}.
\end{proof}

The result of Theorem \ref{6.1}, together with Theorem \ref{main1}, completes our proof of the main result.
\begin{proof}[Proof of Theorem \ref{classification}]
Let $G$ be a closed subgroup with $Aut(\Gamma)^*\leq G<Sym_{\{l, r\}}(\Gamma)$. Then by Theorem \ref{6.1}, $G\leq S_{\{l, r\}}{(\Gamma)}$. Using the result of Theorem \ref{main1}, we have $G=S_X{(\Gamma)}^*$ for some subset $X\subseteq \{l, r\}$. This completes the proof of Theorem \ref{classification}.
\end{proof}

%%%%%%%%%%%%%%%%%%%%%%%%%%%%%%%%%%%%%%%%%%%%%%%%%%%%%%%%%%%%%%%%%%%%%%%%%%%%%%%%%%%%%%%%%%%%%%%%%%%%%%%%%%%%%%%%%%%%%%%%%%%%%%%%%%%%%%%%%%%%%%%%%%%%%%%%%%%%%%%%%%%%%%%%%%%%%%%%%%%%%%%%%%%%%%%%%%%%%%%%%%%%%%%

%%%%%%%%%%%%%%%%%%%%%%%%%%%%%%%%%%%%%%%%%%%%%%%%%%%%%%%%%%%%%%%%%%%%%%%%%%%%%%%%%%%%%%%%%%%%%%%%%%%%%%%%%%%%%%%%%%%%%%%%%%%%%%%%%%%%%%%%%%%%%%%%%%%%%%%%%%%%%%%%%%%%%%%%%%%%%%%%%%%%%%%%%%%%%%%%%%%%%%%%%%%%%%%

\bibliography{reducts_graph}%%
\bibliographystyle{asl}

\end{document}